\documentclass{article}
\usepackage[cp1251]{inputenc}
\usepackage[english]{babel}
\usepackage{latexsym,amsfonts,amsthm}
\usepackage{amssymb}
\usepackage{textcomp}
\usepackage{amsmath}
 \usepackage{amscd,amssymb,epsfig, epsf}
\usepackage{graphicx}
\usepackage{epstopdf}
\newtheorem{lemma}{ Lemma}
\newtheorem{theorem}{ Theorem}
\newtheorem{definition}{ Definition}

\def\endproof{\hfill \vrule height8pt width4pt}

\begin{document}

\begin{center}
\textsl{A.A. Akimova, S.V. Matveev}\footnote{Both authors are
supported by RFBR, project no. 12-01-00748,  Scientific School
      Grant no. 1414.2012.1, and   the joint
      research project 12-C-1-1018-1 of Ural and Siberian branches of RAS}
       \\[5pt]

\textsc{\textbf{ Classification of   knots in $T\times I$ with at
most 4 crossings
 }}\\[15pt]
\end{center}

We compose the  table of knots in the  thickened torus $ T\times I
$   having  diagrams with $\leq 4$ crossings. The knots are
constructed by the three-step  process. First we list regular
graphs of degree 4 with $\leq 4$ vertices, then for each graph we
enumerate all corresponding knot projections, and after that we
construct the corresponding minimal diagrams.  Several  known and
new tricks made it possible to keep  the process  within
reasonable limits and offer a rigorous theoretical proof of the
completeness of the table. For proving that all knots are
different we use a
generalized version of the Kauffman polynomial.\\


\begin{center}\textbf{
Introduction}\\[5pt]
\end{center}

The interest in knots in manifolds of type  $F \times I$, where
$F$ is closed orientable surface, has increased in recent years.
The torus $T=S^1\times S^1$  is the most simple closed orientable
surface after  $S^2$. So the theory of  knots in $T \times I$ is a
natural  generalization of the theory of knots in $S^{2}\times I$,
which is equivalent to the    theory of knots in $S^3$. Knots in
$T\times I$ can be represented by diagrams similar to spherical
diagrams of classical knots. The Reidemeister moves play the same
role: they implement knot isotopies.

First tables of knots had been composed by P. Tait  in 1876 [1].
Then these tables had been   enlarged ([2, 3]). Now there exist
tables of knots in $S^{3}$ having diagrams with $\leq 16$ and even
$\leq 22$ crossings [4, 5]. On the other hand, there are only a
few papers on tabulation of knots in thickened surfaces, see [6,7]
for knots in $RP^2\tilde{\times}I $, which is   the punctured
projective space. An efficient method for tabulating tangles is
described in  [8]. Links in the thickened torus had been studied
in [9, 10]. See also [11].

This paper is devoted to tabulating knots in the thickened torus
$T \times I$ having diagrams with $\leq 4$ crossings. The knots
are constructed by the three-step  process. First we list regular
graphs of degree 4 with $\leq 4$ vertices, then for each graph we
enumerate all corresponding knot projections, and after that we
construct the corresponding minimal diagrams.  Several  known and
new tricks made it possible to keep the process  within reasonable
limits and offer a rigorous theoretical proof of the completeness
of the table. For proving that all knots are different we use a
generalized version of the Kauffman polynomial
[12], see also [13].\\

\textbf{
\center{\S 1. The main result\\[5pt]}}

\begin{definition}Let $T=S^1
\times  S^1$  be the two-dimensional torus and let $I$ be the
interval $[0, 1]$. A knot in $T\times I$ is an arbitrary simple
closed curve $K \subset T\times I$. Two knots $ K, K' \subset T
\times I $ are
 equivalent if the pairs $ (T \times  I, K)$,
 $(T
  \times  I, K')$ are homeomorphic.
  \end{definition}

Knots in $T\times I$, as well as   classical knots, can be
represented by projections and diagrams. By a projection of a knot
in $  T\times I$ we mean a regular graph $G\subset T$ of degree 4
such that the   ``straight ahead'' rule determines a cycle
composed of all the edges of $G$. This cycle can be converted into
a knot diagram by breaking it in each crossing point to show which
strand is going over the other.
  Two projections $G, G'$ are called equivalent if
   the pairs $ (T ,G)$,
 $(T,
    G')$ are homeomorphic.  The diagram
equivalence has the same meaning. In addition we allow
simultaneous crossing change at all crossings.

\begin{definition}
A diagram  of a knot  $K\subset T\times I $ is called minimal if
its complexity (the number of  crossings) is no more than the
complexity of any diagram of any knot   equivalent to $K$. A
projection $G\subset T$ is   minimal if at least one of the
corresponding knot diagrams   is minimal.
  \end{definition}

   We   say that a knot $K\subset T\times I$
   is local if it is contained in a   ball $V\subset  T\times I$
  and   composite, if there is a ball $V\subset  T\times I$ such that
  $\partial V$ decomposes $K$ into two nontrivial arcs in $V$ and the complement of $V$.
         Our table    consists of
     knots which are     prime, that is, neither local nor composite.

\begin{theorem}  There exist exactly 64 different  prime  knots in $T\times
I$ having $\leq 4$ crossings. Diagrams of those knots are
shown in Fig.~\ref{fig1}.
\end{theorem}
\begin{figure}
\centerline{\psfig{figure=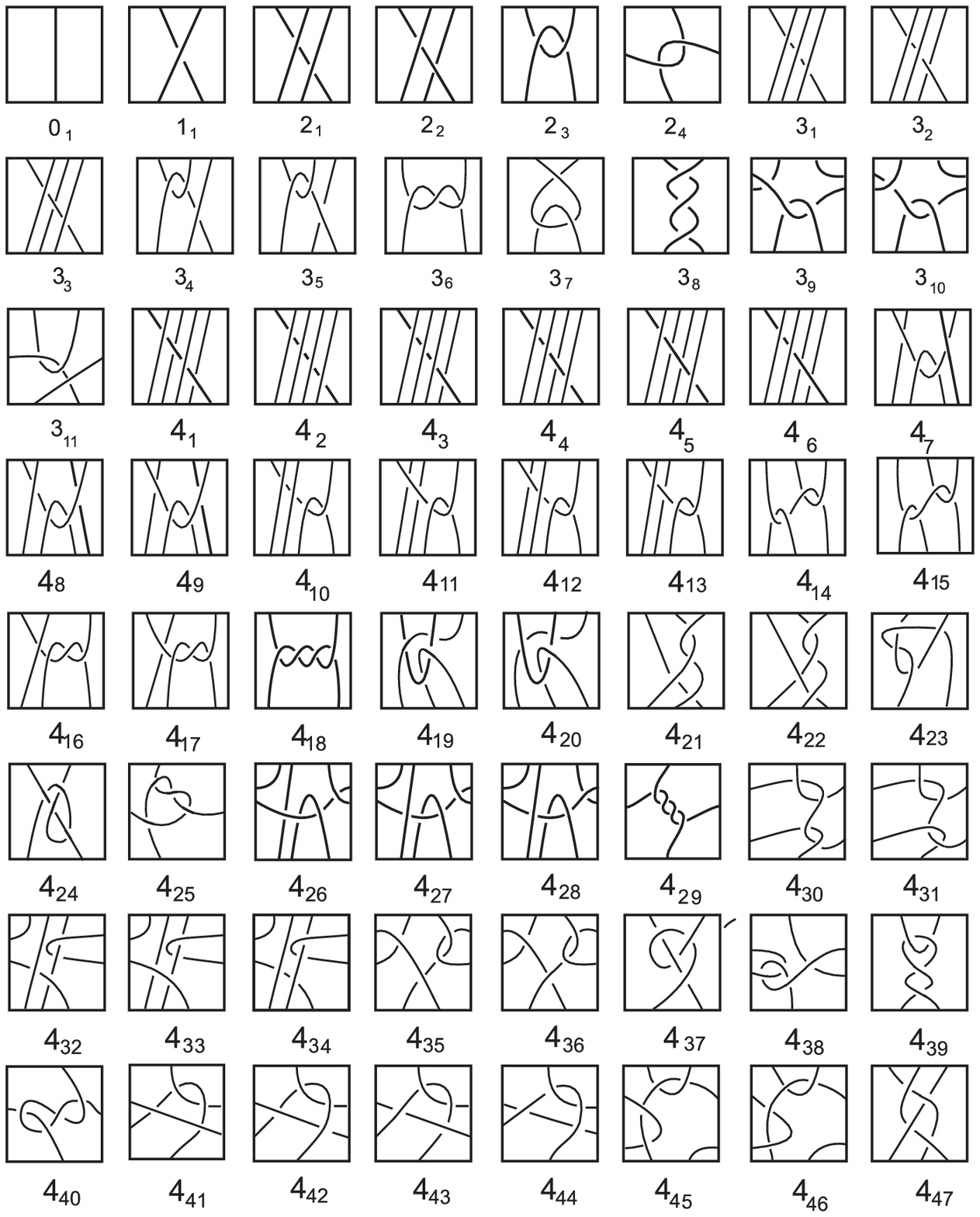,height=18cm}}
  \caption{  }
  \label{fig1}
\end{figure}

The proof of Theorem 1 consists of three steps. First  we list all
abstract regular graphs with $\leq 4$ vertices and classify all
prime  projections in $T$. Then we enumerate all corresponding
diagrams. By performing  this step we use different tricks for
removing duplicates (i.e. diagrams representing equivalent knots).
At the last step we use a generalized version of the Kauffman
polynomial  for proving that all
  knots thus obtained are different.

\textbf{
\center{\S 2. The enumeration of graphs and projections \\[5pt]}}

\begin{lemma}  Any regular graph $G$ with $n\leq 4$
vertices contains   a    loop or a multiple edge.
\end{lemma}
\begin{proof} To the contrary, suppose that $G$ contains  no loops and multiple edges.
Denote by  $N$ the number
 of edges of   $G$. Then $N$ is at most $C_n^2=n(n-1)/2$.
 On the other hand, we have the equality $N=2n$, because $G$ is regular.
    This contradicts the assumption    $n\leq 4$.
 \end{proof}

Note that prime projections have no  trivial loops, since any
trivial loop in a knot diagram can be   removed by   the first
Reidemeister move. Moreover,  any knot projection cannot  contain
 more than two nontrivial loops (otherwise the projection would
be disconnected).

\begin{lemma}There exist exactly 15 regular graphs with $\leq 4$ vertices and $\leq 2$
 loops, including the circle without vertices. See Fig.~\ref{fig2}.\\
\end{lemma}
\begin{figure}[htb]
\centerline{\psfig{figure=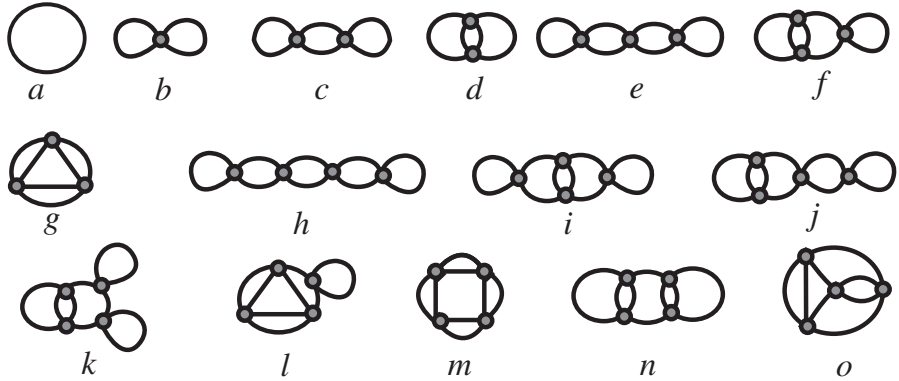,height=5cm}}
  \caption{Regular graphs with $\leq 4$ vertices and $\leq 2$ loops }
  \label{fig2}
 \end{figure}
  \begin{proof} It follows from  Lemma 1 that all regular graphs with $n\leq 4$
  vertices can be obtained from the circle by   $n$ operations of the following   two
  types: 1) insertion of a loop,  and 2) identification of a point  on
   a loop with a point on another  edge. It remains to construct
   all graphs which can be obtained  from the circle by   1, 2, 3 or 4 such
   operations and remove all duplicates and  graphs with $\geq 3$   loops.
\end{proof}

\begin{definition} Suppose that   a projection   $G$ and a disk $D$ in $T$
 are chosen so that
   $G\cap D$ consists of   two disjoint
proper arcs  $l_1,l_2\subset D$. Let $l'_1,l'_2 \subset D$ be two
new arcs such that they have the same endpoints and $l_1' \cap
l_2'$ consists of two transverse points.  Then {\em the biangle
  addition} consists in replacing
 $l_1,l_2 \subset D$ by $l'_1,l'_2 \subset D$.
\end{definition}

For  performing this operation it suffices to choose a simple arc
$\alpha\subset T$   connecting  two non-vertex points  of $G$ and
a regular  neighborhood $D\subset T$ of $\alpha$, see
Fig.~\ref{fig3}. The inverse operation   is called {\em the
biangle removal}.
\begin{figure}[!!!!!!!ht]
\centerline{\psfig{figure=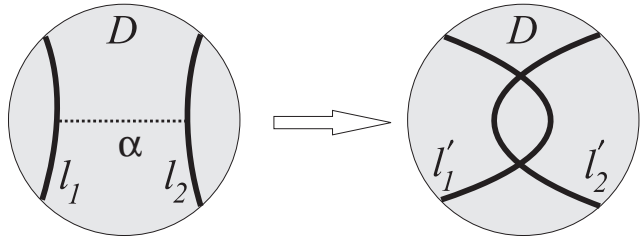,height=3.1cm}}
  \caption{Biangle addition }
  \label{fig3}
\end{figure}

Let $G\subset T$ be a projection. By analogy with knot diagrams we
will say that $G$ is local if it is contained in a disk in $T$,
and composite if there is a disk $D\subset T$ such that neither
$D\cap G$ nor $(T\setminus{\rm Int }D )\cap G$ are trivial arcs.
In particular, prime projections have no trivial loops.

\begin{theorem} There exist exactly 36 different prime projections
  in $T$ with $\leq 4$ crossings, see Fig.~\ref{fig4}. All these projections are minimal.\\
\end{theorem}
\begin{figure}[htb!!!!!!!!!!!!]
\centerline{\psfig{figure=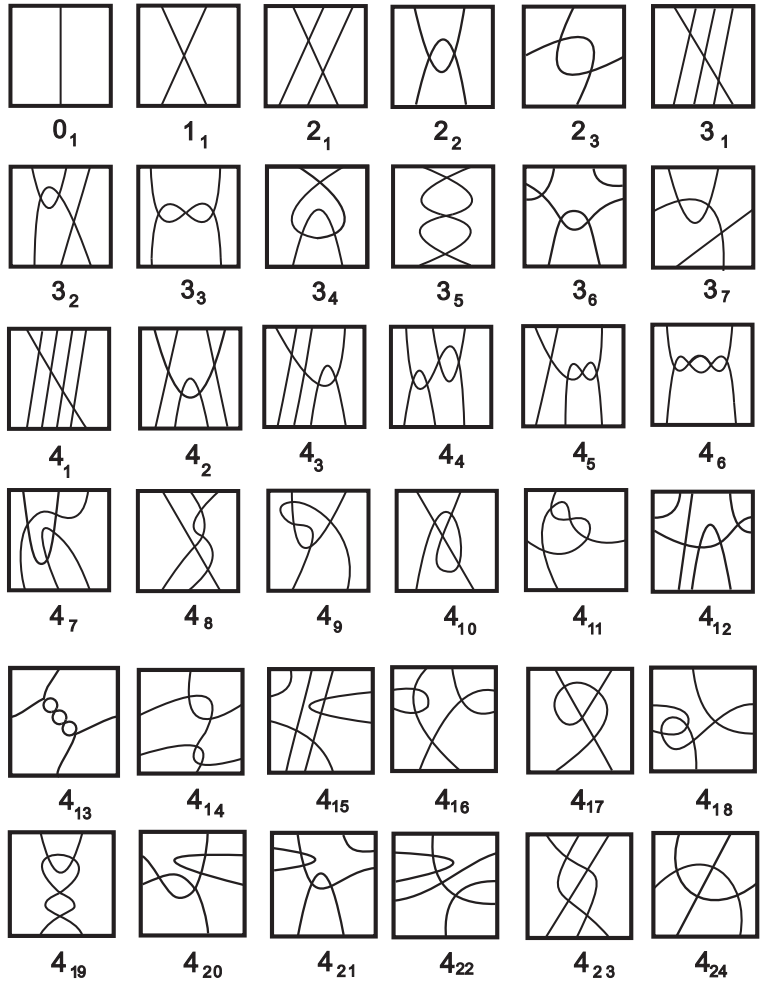,height=13cm}}
  \caption{Projections in $T$, which is represented as a square with
  identified opposite sides}
 \label{fig4}
\end{figure}

 \begin{proof}
   It is easy to see that any  projection without vertices
    (i.e. any nontrivial circle in $T$)
   is equivalent to the projection $\bf 0_1$ in Fig.~\ref{fig4}.
     Let $G$ be a prime projection
corresponding to one of the graphs $b$, $c$,  $e$, and $h$. We
   cut   $G$ at all vertices so as to produce   $n+1$
 disjoint circles, where
$n$ is the number of  vertices  of $G$. Circles corresponding to
the loops of $G$ are nontrivial. It follows that they are parallel
in $T$. Each of the remaining circles may be
  either trivial or not. We list  all
combinations of their types and obtain  projections  $ \bf 1_1$,
$\bf 2_1$ -- $\bf 2_2$,   $\bf 3_1$ -- $\bf 3_3$, $\bf 4_{1}$ --
$\bf 4_{6}$ for $n=1,2,3,4$, respectively.

Let us consider a prime projection $G$ of type  $d$. It is the
union of two circles with two common  points. With respect to the
way how $G$ lies in $T$, exactly one of these points is
transverse, since otherwise we would have a link of two
components.   Let us cut $G$ at the  non-transverse point so as to
get two circles $\mu, \lambda$ with one transverse crossing point.
The complement to the circles in $T$ is a disk. It follows that
there is only one way to perform the inverse operation, i.e.
identification of a point of $\mu$ with a  point of $\lambda$. The
identification produces a projection equivalent to $\bf 2_{3}$.

Let us prove that there are no prime projections corresponding to
the graphs $f,  i, j, k$. Indeed, every such projection $G$ can be
obtained from the  projection $\bf 2_{3}$ by attaching one or two
loops. The complement to   $\bf 2_{3}$ in $T$ consists of two
disks. It follows that the loops   are trivial, in contradiction
with our assumption that $G$ is prime.

 Let us enumerate projections corresponding to the graph $g$,
 which consists of
   three circles such that every circle has one common point with each of
   the two other circles. Suppose that the common point of two circles
   (denote them $\mu, \lambda$)
     is
   transverse. As before, the  complement  to $\mu\cup  \lambda$ in $T$ consists of two disks.
   Therefore there is  only one way
to place the third circle into $T$.  We get the projection $\bf
3_{7}$.

Suppose that in all three common points of  the circles the
intersection is non-transverse.   Each circle can be either
trivial or not. Investigation of   all possibilities  shows that
one can get a projection of a knot  only in the case when the
number of trivial circles is odd (3 or 1). Moreover, in the  first
case
  $G$ must be contained in an annulus in $T$ having connected complement.
   This gives the  projection $\bf 3_5$. In the
second case we get   projections $\bf 3_4$ and $\bf 3_6$,
depending on   how a nontrivial circle and the trivial one are
approaching  to the second nontrivial circle: from the same side
or not.

Let us list all projections of type $l$. Each of them can be
obtained from some projection $G$ of type $g$
    by   inserting  a loop, which must be nontrivial. Suppose that
    $G$ is prime. Hence it  is one of the projections $\bf 3_4$, $\bf 3_5$.
  Since the complements to  $\bf 3_{6}$
and $\bf 3_7$ consist of disks,   one cannot insert a nontrivial
loop. There is only one way to add a nontrivial loop to $\bf
3_{4}$ as well as to  $\bf  3_5$. Doing so we get
  projections $\bf 4_7$ and $\bf 4_8$.

  Suppose that $G$ is not prime. Then it is either the standard projection of a local trefoil,
  or a projection composed from a local trefoil and  a nontrivial circle
   embedded in $T$. In the first case inserting a loop gives us a
   composite projection, in the second one we get
two projections $\bf 4_{9}$  and $\bf 4_{10}$ of type $l$.

We construct   projections of type $m$ in the same way as the
projections of type $g$. Any projection of type $m$  consists of
four circles such that each circle has two common point with the
other three circles.  Suppose that at least one of the common
point of two circles
   (denote them $\mu, \lambda$)     is   transverse.  As before, the  complement
    to $\mu\cup  \lambda$ in $T$ is a disk.   Therefore there is  only one way
to place   two other circles into $T$.  We get the projection $\bf
4_{11}$.

Suppose that at all four common points of  the circles the
intersection is non-transverse.   Each circle can be either
trivial or not. As in the case of  graph $g$, investigation of all
possibilities shows that one can get a projection of a knot  only
in the case when the number of trivial circles is odd (1 or 3). We
get the projection $\bf 4_{12}$ in the first case  and the
projection $\bf 4_{13}$ in the second one.

Note that any  projection $G$ of type $n$ contains two triple
edges.
 Then the
projection $G'\subset T$ obtained from $G$ by removing two edges
$e_1, e_2$ from one of those triple edges is of type $d$ and thus
is equivalent to the projection $\bf 2_3$. The complement to $G'$
in $T$ consists of two disks. In order to restore $G$   we should
add
  $e_1,e_2$ to $G'$ so that they   have common endpoints
on an edge $e$   of $G'$. One of those two edges must approach to
$e$ from the same side while the other  from the different sides.
Otherwise we obtain a   projection of a link. Therefore there is
only one way to add $e_1,e_2$ to $G'$, which gives   the
projection $\bf 4_{14}$.

Now we consider a projection $G$ of the last type $o$. Suppose
that one of the faces of  $G$ is a biangle. Let us  remove this
biangle by the operation in Fig.~\ref{fig3}, see Definition 2. We
get the projection which has  two vertices and thus corresponds
  to one of the graphs $c$ or $d$. It follows
that one can get $G$ from a projection of type $c$ or $d$ by the
biangle addition. The idea of the next part of the proof is to
list all projections of types $c$ and $d$ and to look along which
arcs one  can add biangles to them   so as to get prime
projections of type $o$. In the case of the type $c$ projection we
may consider only   arcs connecting loops.

All projections of types $c$ and $d$ with   arcs producing prime
projections
 of type $o$ are represented in Fig.~\ref{fig5}.
In the cases 1 -- 6 we get the projections $\bf 4_{15}-4_{20}$. In
cases 7 -- 10 we get two new projections of type $o$: $\bf 4_{21}$
in
  case 7 and $\bf 4_{22}$ in   case 8. Cases 9, 10 give
projections $\bf 4_{19}$ and $\bf 4_{20}$ obtained earlier.

\begin{figure}[htb]
\centerline{\psfig{figure=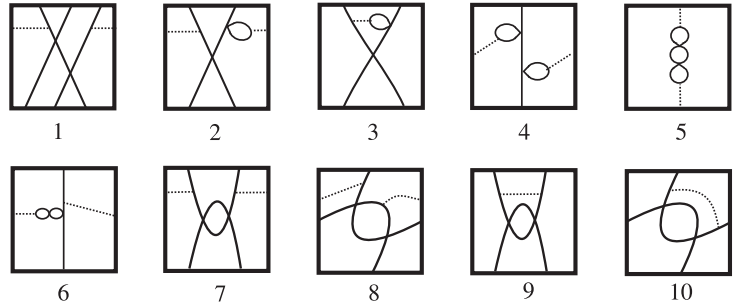,height=4cm}}
  \caption{How one can obtain all projections of type $o$}
  \label{fig5}
 \end{figure}

Now suppose that $G$ has  no biangle faces. Then its double edges
define disjoint nontrivial circles in $T$. Each of the  circles
contains two vertices  and their union decomposes  $T$ into two
annuli. The remaining four edges must be contained in these annuli
($m$ in the first annulus, $n$ in the second one) and connect
every vertex of the first circle with every vertex of the second
one. There are three possibilities: $(m,n)=(0,4),(1,3)$ and
$(2,2)$. We get the projections $\bf 4_{23}$, $\bf 4_{24}$ in the
first and   second cases, and a link projection in the third one.

Let us prove that all projections in Fig.~\ref{fig4} are
different. To this end we count the angles in each disc or annular
face of the projection.
    It turns out that the  sets of numbers thus obtained
    (with specifying the numbers corresponding to annular faces)
    are different for  all projections in Fig.~\ref{fig4},
except  projections $\bf 4_{11}$, $\bf 4_{21}$. These projections
have the same   set $ \{2, 2, 4, 8\}$ and  are different since the
biangle faces have a common vertex in the first case and no common
vertices  in the second one. Also all 36 projections are minimal,
see the last sentence of the proof of Theorem 1.
\end{proof}

\textbf{
\center{\S 3.    Proof of Theorem 1\\[5pt]}}

 We recover all prime knots having   projections with $n \leq
4 $ crossing points  (see Theorem 2 and Fig.~\ref{fig5}) by
indicating the types of  crossings  using  all $2^n$ possible
ways.  However one can essentially reduce this procedure by using
the following ideas.

 \begin{enumerate}
 \item The
simultaneous crossing changes in all crossings converts any
diagram to an equivalent one.   Therefore for any projection with
   $n$ vertices it suffices to consider $2^{n-1}$ possibilities.

 \item  Let $x,y$ be the vertices of a biangle face of a given projection.
 Then there are only two  ways of indicating the crossing
 types at $x,y$
  such that one  cannot remove the biangle by the second
 Reidemeister move. Therefore every biangle reduces the procedure   in two times.

 \item There are only two ways to indicate the under- and over-crossings
 for the fragment shown  in the center of Fig.~\ref{fig6}. All other
  ways give us non-minimal diagrams.

\item Let a given projection contain a triangle  face. Suppose
that the types of its  vertices    are chosen  so that we can
perform the third Reidemeister move and get a diagram
constructed earlier.
 Then one  may drop this choice.

\end{enumerate}

 \begin{figure}[htb]
\centerline{\psfig{figure=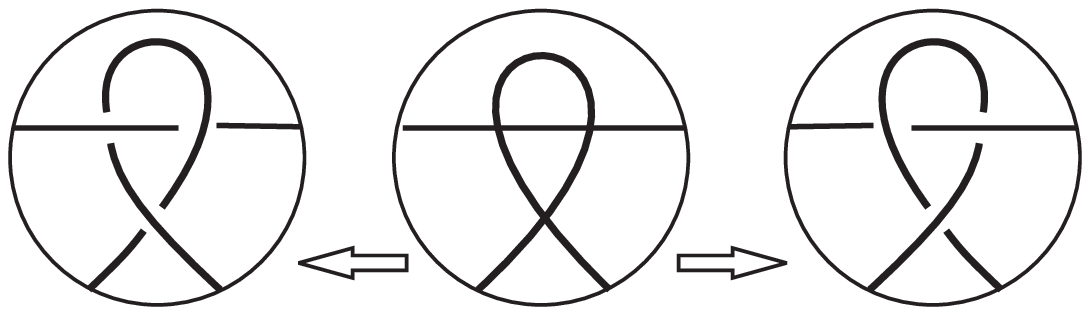,height=2cm}}
  \caption{ }
  \label{fig6}
 \end{figure}

Those ideas are enough for removal almost all duplicates. The only
exception is the diagram in Fig.~\ref{fig7} (left). It can be
transformed  into the diagram $\bf 4_{31}$ (right) we have
constructed before.
\begin{figure}[htb]
\centerline{\psfig{figure=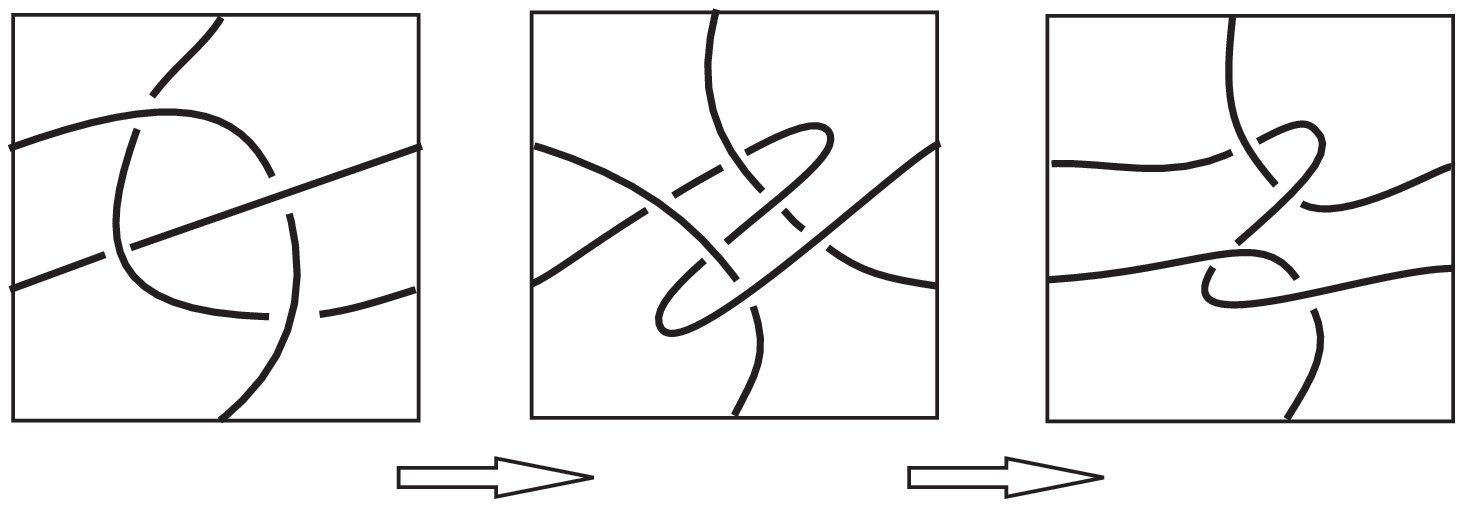,height=3cm}}
  \caption{ }
  \label{fig7}
 \end{figure}

All knots in  Fig.~\ref{fig1} are distinct. This can be   proved
by using the  generalized Kauffman polynomial, which
 is slightly  different from the usual normalized Kauffman bracket  [9, 10].
  We use two variables $a$ and $x$. These two variables are
needed for taking into account the number and the types (trivial
or not) of the circles in
 $T$ which we   obtain  after resolving all crossings.
 In addition we use a
non-standard normalization: we set the  polynomial of the trivial
knot to be
 $(-a^2-a^{-2})$,   not 1 as   usual. This is because quite often
 resolving all crossings produce
 no trivial circles. In those cases we would be forced to
 divide the polynomials we have obtained by $(-a^2-a^{-2})$, which is inconvenient.
   The exact formula is the following:
$$X(K)=(-a)^{-3w(K)}\sum_sa^{\alpha(s)-\beta(s)}(-a^2-a^{-2})^{\gamma(s)}x^{\delta(s)},$$
where $\alpha(s)$ and $\beta(s)$ are the numbers of markers $A$
and $B$ in a given state  $s$,  and $\gamma(s)$, $\delta(s)$ are
the numbers of trivial and nontrivial circles in $T$ obtained by
resolving all crossing points. Just as for the original Kauffman
polynomial,  the sum is taking over all   states. Of course,
$w(K)$ is the writhe of the diagram.   The next table was obtained
by direct
 calculations. One can see that all  polynomials are distinct.
Therefore all knots in Fig. 1 are also distinct.   It follows
that  all 36 projections in Fig.~\ref{fig2} are
minimal.$\endproof$

\vspace{0,3cm}

 \noindent
$0_{1}$: $x$\\[3pt]
$1_{1}$: $-(x^2a^{-4}-a^{-4}-1)$\\[3pt]
$2_{1}$: $x^3a^{-8}+x(-2a^{-8}-a^{-4})$\\[3pt]
$2_{2}$: $x^3+x(-a^{-4}-1-a^{4})$\\[3pt]
$2_{3}$: $x^2(a^{-6}-a^{-10})-a^{-6}-a^{-2}$\\[3pt]
$2_{4}$: $x(a^{-8}+a^{-6}-a^{-2})$\\[3pt]
$3_{1}$: $-(x^4a^{-12}+x^2(-a^{-8}-3a^{-12})+a^{-8}+a^{-12})$\\[3pt]
$3_{2}$: $-(x^4a^{-4}+x^2(-a^{-8}-2a^{-4}-1)+a^{-8}+a^{-4})$\\[3pt]
$3_{3}$: $-(x^4a^{-4}+x^2(-2-a^{-4}-a^{-8})+1+a^{4})$\\[3pt]
$3_{4}$: $-(x^3(a^{-10}-a^{-14})+x(-a^{-10}-a^{-6}+a^{-14}))$\\[3pt]
$3_{5}$: $-(x^3(a^{-2}-a^{-6})+x(a^{-10}-a^{-2}-a^{2}))$\\[3pt]
$3_{6}$: $-(x^2(a^{-16}+a^{-8}-a^{-12})-a^{-8}-a^{-4})$\\[3pt]
$3_{7}$: $-(x^2(a^{-8}-a^{-4})+a^{-4}+a^{-16})$\\[3pt]
$3_{8}$: $-(x^2a^{-12}-a^{-12}-1)$\\[3pt]
$3_{9}$: $-(x^2(a^{-8}-a^{-4}+a^{-12})-a^{-8}-2a^{-12}+a^{-4})$\\[3pt]
$3_{10}$: $-(x^2(2a^{-4}-1)-2a^{-4}-a^{-8}+a^{4})$\\[3pt]
$3_{11}$: $-x(-a^{-6}-a^{2}-a^{-8}+a^{-4}+a^{-2})$\\[3pt]
$4_{1}$: $x^5+x^3(-1-2a^{4}-2a^{-4})+x(1+a^{8}+a^{4}+a^{-4}+a^{-8})$\\[3pt]
$4_{2}$: $x^5a^{-16}-x^3(a^{-10}+4a^{-16})+x(2a^{-12}+3a^{-16})$\\[3pt]
$4_{3}$: $x^5a^{-8}-x^3(a^{-12}+3a^{-8}+a^{-4})+x(a^{-4}+2a^{-12}+2a^{-8})$\\[3pt]
$4_{4}$: $x^5a^{-8}+x^3(-2a^{-4}-2a^{-8}-a^{-12})+x(a^{-12}+2a^{-4}+a^{-8}+1)$\\[3pt]
$4_{5}$: $x^5+x^3(-2a^{4}-2-a^{-4})+x(a^{8}+2a^{4}+2)$\\[3pt]
$4_{6}$: $x^5-x^3(a^{4}+3+a^{-4})+x(a^{4}+3+a^{-4})$\\[3pt]
$4_{7}$: $x^4(-a^{-2}+a^{2})-x^2(a^{6}-2a^{-6}+a^{2})-a^{-10}-a^{-6}$\\[3pt]
$4_{8}$: $x^4(a^{-14}-a^{-18})-x^2(a^{-10}-2a^{-18}+a^{-14})-a^{-18}-a^{-14}$\\[3pt]
$4_{9}$: $x^4(-a^{-10}+a^{-6})+x^2(-a^{-2}+a^{-14}+a^{-10}-a^{-6})-a^{-14}-a^{-10}$\\[3pt]
$4_{10}$: $x^4(-a^{-18}+a^{-14})+x^2(-a^{-10}+2a^{-18}-2a^{-14})+a^{-10}+a^{-14}$\\[3pt]
$4_{11}$: $x^4(-a^{-2}+a^{2})+x^2(-a^{6}+a^{-6}+a^{-2}-2a^{2})+a^{6}+a^{2}$\\[3pt]
$4_{12}$: $x^4(a^{-6}-a^{-10})+x^2(-a^{-6}+a^{-14}-a^{-2})+a^{-6}+a^{-10}$\\[3pt]
$4_{13}$: $x^4(a^{-6}-a^{-10})+x^2(a^{-14}-2a^{-4})+a^{-2}+a^{2}$\\[3pt]
$4_{14}$: $x^3(2-a^{-4}-a^4)+x(-1+a^{8}+a^{-8})$\\[3pt]
$4_{15}$: $x^3(a^{-12}-2a^{-16}+a^{-20})+x(-a^{-8}+2a^{-16})$\\[3pt]
$4_{16}$: $x^3(a^{-12}-a^{-16}+a^{-20})+x(-a^{-8}-a^{-12}-a^{-20})$\\[3pt]
$4_{17}$: $x^3(a^{-4}-a^{-8}+a^{-12})-x(a^{-10}+a^{-4}+a^{-16})$\\[3pt]
$4_{18}$: $x^2(a^{-10}-a^{-14}+a^{-18}-a^{-22})-a^{-6}-a^{-10}$\\[3pt]
$4_{19}$: $x^3(a^{-12}-a^{-8})+x(-a^{-16}-a^{-20}+a^{-4})$\\[3pt]
$4_{20}$: $x^3(a^{-4}-1)+x(2-a^{-4})$\\[3pt]
$4_{21}$: $x^3a^{-16}-x(2a^{-16}+a^{-8})$\\[3pt]
$4_{22}$: $x^3a^{-8}-x(a^{-12}+a^{-8}-1+a^{4}+a^{-4})$\\[3pt]
$4_{23}$: $x^2(a^{-14}-a^{-30}-a^{-10}+a^{-6})-a^{-2}-a^{-24}$\\[3pt]
$4_{24}$: $x^2(a^{-6}+a^{2}-2a^{-2}-a^{-10})+a^{-14}+a^{-2}$\\[3pt]
$4_{25}$: $x(a^{-20}+a^{-18}-a^{-16}-a^{-14}+a^{-12}+a^{-10}-a^{-6})$\\[3pt]
$4_{26}$: $x^3(a^{-10}-a^{-6})+x(2a^{-6}-2a^{-10}+a^{-16})$\\[3pt]
$4_{27}$: $x^3(a^{-6}-a^{-2})+x(a^{2}+a^{-8}-a^{-10})$\\[3pt]
$4_{28}$: $x^3(a^{-2}-a^{2})+x(a^{6}+a^{2}+1-a^{-2}-a^{-6})$\\[3pt]
$4_{29}$: $x(a^{-16}-a^{-2}+a^{-6}+a^{-14}-a^{-38})$\\[3pt]
$4_{30}$: $x(a^{-4}+a^{-6}+a^{-16}-2a^{-8}-2a^{-10}+2a^{-14})$\\[3pt]
$4_{31}$: $x(a^{-2}-a^{6}-a^{4}+a^{2}+3-a^{-4}-a^{-6})$\\[3pt]
$4_{32}$: $x^3a^{-8}-x(a^{-2}+a^{-4}-a^{-6}+a^{-8}+a^{-12})$\\[3pt]
$4_{33}$: $x^3+x(a^{8}-2a^{4}-a^{2}-1-a^{-2}-a^{-4})$\\[3pt]
$4_{34}$: $x^3a^{-16}+x(-a^{-6}-a^{-8}+a^{-10}+a^{-12}-3a^{-16})$\\[3pt]
$4_{35}$: $x^2(-a^{-6}-a^6+a^{2})+a^{6}+2a^{-6}-a^{2}$\\[3pt]
$4_{36}$: $x^2(2a^{-6}-2a^{-10}-a^{-2})-a^{-6}+a^{-14}+a^{-10}+a^{2}$\\[3pt]
$4_{37}$: $x^2(-a^{-6}-a^{2}+a^{-2})+a^{-6}-a^{10}+a^{6}+a^{2}$\\[3pt]
$4_{38}$: $x(-a^{-6}-a^{-8}+a^{-10}+2a^{-12}-a^{-16}+a^{-20})$\\[3pt]
$4_{39}$: $x^2(a^{2}-a^{6})-a^{6}+a^{2}+a^{-2}+a^{-10}$\\[3pt]
$4_{40}$: $x(-a^{-18}-a^{-16}+a^{-14}+2a^{-12}+a^{-4}-a^{-8})$\\[3pt]
$4_{41}$: $x(a^{8}-2a^{4}-a^{2}+1+2a^{-2})$\\[3pt]
$4_{42}$: $xa^{-12}-a^{-2}+a^{-4}-a^{-6}-a^{-8}-a^{-10}$\\[3pt]
$4_{43}$: $x(2a^{-6}-a^{-2}+a^{-8}-a^{-12})$\\[3pt]
$4_{44}$: $x(-a^{-6}-a^{-8}+a^{-10}+2a^{-12}+a^{-14}-a^{-16})$\\[3pt]
$4_{45}$: $x^2(2a^{-14}-2a^{-10})-a^{-2}+a^{-6}+a^{-10}-2a^{-14}-a^{-18}$\\[3pt]
$4_{46}$: $x^2(-a^{6}+a^{2}+a^{-2}-a^{-4})+a^{6}-2a^{2}-2a^{-2}+a^{-6}$\\[3pt]
$4_{47}$: $x^3+x(-1+a^{8}-2a^{4}-2a^{-4}+a^{-8})$\\[3pt]

\center{\textbf{\S 4. Final remarks \\[5pt]}}
 \begin{enumerate}
 \item The degrees of  $x$ in every polynomial have the same parity.
  This parity depends only on the type (trivial or nontrivial) of the element $[K]$ of
   $H_1(T;\mathbb{Z}_2)$ corresponding to the given knot $K$. One
   can easily see that $[K]=0$ if and only if any diagram of $K$
   crosses each side of the square in an even number of points.

\item  The table contains exactly 10 homologically trivial knots:
$2_3$, $3_7$, $4_7 - 4_9, 4_{18}, 4_{23}, 4_{39}, 4_{45}, 4_{46}$.

  \item The table contains exactly 23 alternating diagrams:
 $1_1$, $2_2, 2_3$, $3_3, 3_5-3_8, 3_{10}$,
$4_1, 4_7, 4_{13}, 4_{14}, 4_{17}-4_{19}, 4_{22}-4_{24}, 4_{36},
4_{37}, 4_{39},  4_{45}$. As it should be, each of the corresponding knots has exactly one minimal diagram.

\item About a half of 36 projections from Theorem 2 determine
 only one  knot each. The projection $\bf 4_1$ determine  the maximal number
of knots (6). The average number of knots for one projection is
about 1,8.
\end{enumerate}

{\bf References}
\begin{enumerate}
\item    \textit{Tait, P. G.} On knots I, II, III // Cambridge
University Press. 1898 -- 1900. Including Trans. Roy. Soc.
Edinburgh. -- 1877. -- 28.-- P. 35 -- 79.

 \item  \textit{Alexander, J.W.} On types of knotted curves //
Ann. Math. -- 1927. -- 28. -- P. 562 -- 586.

\item   \textit{Rolfsen, D.}  Knots and Links // Mathematics
Lecture. Publish or Perish, Inc., Berkeley, Calif.-- 1976. --
7.-- ix+439 pp.

\item   \textit{Hoste, J., Thistlethwaite, M., Weeks, J. }
 The first 1,701,935
knots // Math. Intelligencer (Springer). --1998. -- Vol 20, №4. --
P. 33–48.

\item   \textit{Hoste, J.}  The enumeration and classification of
knots and links //    Handbook of Knot Theory, Amsterdam,
Elsevier. -- 2005.- 32 pp.

\item   \textit{Drobotukhina, Yu.V.}  Jones polynomial analog for
links in $RP^3$ and the generalization of Kauffman -- Murasugi
theorem // Algebra and analysis. -- 1991. -- Vol. 2, №3. -- P. 613
-- 630

\item    \textit{Drobotukhina, Yu.V.}   Classification of links in
$ \mathbb{R}P^3$ with at most six crossings // Advances in Soviet
Mathematics. -- 1994. -- Vol. 18, № 1. -- P. 87 -- 121.

\item   \textit{Bogdanov, A.,  Meshkov, V.,  Omelchenko, A.,
Petrov, M.} Enumerating the k-tangle projections // J. of Knot
Theory and its Ramifications.-- 2012. -- Vol. 21, № 7, 17 pp.

\item \textit{Grishanov, S.,  Meshkov, V.,  Omelchenko, A.}
Kauffman-type polynomial invariants for doubly periodic structures
// J. of Knot Theory and its Ramifications. -- 2007. -- Vol. 16,
№ 6. -- P. 779 -- 788.

\item \textit{Grishanov, S., Meshkov, V., Vassiliev, V.}
Recognizing textile structures by finite type knot invariants //
J. of Knot Theory and its Ramifications, 18:2 (2009), P. 209
–-235.

\item \textit{Grishanov, S., Vassiliev, V.}  Fiedler type
combinatorial formulas for generalized Fiedler type invariants of
knots in $M^2\times  R^1$ // Topology and its Applications, Vol.
156 (2009), P. 2307–2316.

\item     \textit{Kauffman, L.} State models and the Jones
polynomial // Topology. -- 1987. -- Vol. 26, №3. -- P. 395 -- 407.

 \item   \textit{Prasolov, V.V.,  Sossinsky, A.B.} Knots, links, braids and 3-manifolds //
    -- M.:
 1997. -- 352 pp.
\end{enumerate}

\vspace{ 1cm}

  Akimova A. A.  \hspace*{3cm}  Matveev S. V.

  \vspace{0,1cm}

   South Ural State University \hspace*{1cm} Chelyabinsk State University

  \hspace*{ 6cm}and IMM of Ural branch of RAS

\vspace{0,1cm}

 akimova\_susu@mail.ru \hspace*{2,5cm}
matveev@csu.ru

\flushleft{}

\end{document}